\title{A Metastability Result for the Contact Process on a Random Regular Graph}
\author{
Wei Su\\
Department of Statistics\\
The University of Chicago\\
}
\date{\today}

\documentclass[12pt]{article}
\usepackage{amsfonts}
\usepackage{amsmath}
\usepackage{cases}
\usepackage{amsthm}
\usepackage{amssymb}
\usepackage{fancyhdr}
\usepackage{indentfirst}
\usepackage{float}
\usepackage{paralist}
\usepackage[margin=1 in]{geometry}

\newtheorem{Theorem}{Theorem}[section]
\newtheorem{Corollary}[Theorem]{Corollary}
\newtheorem{Lemma}[Theorem]{Lemma}

\newtheorem{Remark}[Theorem]{Remark}
\newtheorem{Proposition}[Theorem]{Proposition}

\begin{document}
\maketitle

\begin{abstract}

In this paper we study the metastability of the contact process on a
random regular graph. We show that the extinction time of the contact
process, when initialized so that all vertices are  infected at time $0$,
grows exponentially with the vertex number. Moreover, we show that the
extinction time divided by its mean converges to a unit exponential
distribution in law.

\end{abstract}

\section{Introduction} The contact process $(\xi_t)_{t\geq 0}$ with
infection parameter $\lambda$ on a connected, locally finite graph
$G=(\mathcal{V}_G,\mathcal{E}_G)$ is a continuous-time Markov chain
that evolves as follows. For each $t$ the random variable $\xi_t$
takes value in $\{\text{\rm subsets of }\mathcal{V}_G\}$; we regard
the elements of $\xi_t$ as infected vertices.  Each infected vertex
recovers with rate 1; each healthy vertex (a vertex in $\mathcal{V}_G\backslash
\xi_t$) becomes infected at rate $\lambda$ times the number of
infected neighbors.

The contact process on a finite graph $G$ will eventually reach the
absorbing state $\emptyset$. Of natural interest is the \emph{time to
extinction} $\tau_{G}$ (defined to be the time of the first visit to
$\emptyset$) when the contact process is started from the
full-occupancy state. Since large connected graphs $G$ will contain
long linear chains, when the infection parameter $\lambda$ exceeds the
critical value $\lambda_{c}$ for the contact process on $\mathbb{Z}$ the
contact process on $G$ can be expected to survive for a long time,
eventually reaching a quasi-stationary state that persists until, by
chance, a large number of infected vertices almost simultaneously
become healthy, leading to subsequent extinction. This phenomenon has
become known generically as ``metastability''.  Metastability for
contact process was first introduced in \cite{CGOV} for the
one-dimensional finite cube. Subsequently, there have been various
studies of metastability on other class of graphs, for example, linear
chains in \cite{Schonmann}, the $d-$dimensional finite cube in
\cite{Mountford}, power law random graphs in \cite{MVY}, finite trees
in \cite{CMMV}, and a family of finite graphs in \cite{MMVY}.

In \cite{Lalley-Su} the contact process on random $d$-regular graphs ($d\geq 3$)
is studied. A random regular graph $G\sim\mathcal{G}(n,d)$ is a graph
chosen uniformly from all $d$-regular graphs with $n$ vertices. Such a
graph locally looks like a tree, but globally it differs significantly
from a finite tree. Nevertheless, certain techniques and results that
have been developed for studying the contact processes on trees will be
of use here. In \cite{Liggett2,Pemantle,Stacey} it is shown that when
$d\geq 3$, there exist constants
$0<\lambda_1(\mathbb{T}^d)<\lambda_2(\mathbb{T}^d)<\infty$ that
demarcate different phases for the contact process on an infinite
$d-$regular tree $\mathbb{T}^{d}$. If $\lambda<\lambda_1(\mathbb{T}^d)$,
the contact process started from a finite initial configuration will
almost surely die out; when $\lambda>\lambda_2(\mathbb{T}^d)$ it has
positive probability of local survival; and when $\lambda$ is in
between then it with positive probability survives globally but almost
surely dies out locally.

In this paper we investigate the metastable behavior of the contact
process on random regular graphs $G\sim\mathcal{G}(n,d)$ with
infection parameter
$\lambda>\lambda_1(\mathbb{T}^d)$. Denote by $(\xi^A_t)_{t\geq 0}$  the
contact process on $G$ with initial configuration $A\subset
\mathcal{V}_G$. When $A=\mathcal{V}_G$ we use $(\xi_t)_{t\geq 0}$ as
shorthand; if $A=\{u\}$, we write $\xi^u_t$ instead  of
$\xi^{\{u\}}_t$. Since the underlying graph $G\sim \mathcal{G}(n,d)$
is random, we say that a  property holds for \emph{asymptotically
almost every} $G$ if the set of graphs in $\mathcal{G}(n,d)$ which
satisfy the property has probability  tending to 1 as $n\rightarrow \infty$.

Throughout this paper we fix $\lambda>\lambda_1(\mathbb{T}^d)$, and
we let $p_{\lambda}>0$ be the chance that $\zeta_t^O$, a contact process with initial state $\{O\}$ on $\mathbb{T}^d$, survives
forever. To emphasize conditional probability and expectation given
the graph $G$ we use notations $\mathbb{P}_G$ and $\mathbb{E}_G$. The
word ``typical'' in this paper means asymptotically almost every. Also,  all $o(1)$ terms tend to
$0$ uniformly in $n$ (independent of $G$).

The main results of this paper are Theorem \ref{3.main} and \ref{4.main}.

\begin{Theorem}\label{3.main}
There exists $\beta>0$ such that for asymptotically almost every $G\sim\mathcal{G}(n,d)$,
\[\mathbb{P}_G\{\xi^{}_{\exp(\beta n)}\neq \emptyset\}=1-o(1),\]
where $\{\xi^{}_t\}_{t\geq 0}$ is the contact process with initial configuration $\mathcal{V}_G$.
\end{Theorem}

\begin{Theorem}\label{4.main}
For asymptotically almost every $G\sim\mathcal{G}(n,d)$, the
distribution of $\tau_G/\mathbb{E}\tau_G$ converges to an exponential
distribution with mean 1, where $\tau_G=\inf \{t\geq 0:
\xi_t=\emptyset \}$ is the extinction time of the contact
process started from $\mathcal{V}_G$.

\end{Theorem}

These theorems are proved in sections 2 and 3, respectively. While
preparing this paper, we learned that J.-C. Mourrat and D. Valesin \cite{mourrat-valesin}
have independently established Theorem~\ref{3.main}. Because our
proof is somewhat different from theirs, and because
Theorem~\ref{3.main} is a key complement to Theorem~\ref{4.main}, we
include it  in section 2. 

\section{Exponential extinction time}

The goal of this section is to prove Theorem \ref{3.main}.

Here are some important facts. Let $\zeta^O_t$ be a contact process 
on $\mathbb{T}^d$ with initial state $\{O\}$. 
In \cite{Madras-Schinazi,Morrow-Schinazi-Zhang} it is shown that
\begin{Theorem}
There exist constants $c_\lambda,C(d) \in \mathbb{R}$ such that
\[   e^{c_\lambda t}    \leq \mathbb{E}|\zeta^O_t|\leq C(d)e^{c_\lambda t}.\]
Moreover, if $\lambda>\lambda_1(\mathbb{T}^d)$ then $c_\lambda >0$.
\end{Theorem}
Throughout this paper $c_\lambda>0$ will be the constant in the above theorem.

In \cite{Lalley-Su} (Proposition 5.2 and 5.3) it is shown that
\begin{Proposition}\label{2.typical survival}
Fix $0<\varepsilon<1/8$. For asymptotically almost every $G\sim\mathcal{G}(n,d)$, there are
at least $(1-o(1))n$ vertices (call them ``good'' vertices) in $G$, such
that for each good vertex $u$,
\[\mathbb{P}_G\{\xi^u_{(1+\varepsilon)\log n/c_\lambda}\neq \emptyset\}=(1-o(1))p_{\lambda}.\]
\end{Proposition}
\begin{Proposition}\label{2.typical survival size}
Fix $\delta>0$ and  $0<\varepsilon<1/8$. Then for asymptotically almost every $G\sim\mathcal{G}(n,d)$, if 
$u$ is a good vertex in Proposition \ref{2.typical survival}, then
\[\mathbb{P}_G\{(1-\delta)np_{\lambda}\leq |\xi^u_{(1+\varepsilon)\log n/c_\lambda}|\leq (1+\delta)np_{\lambda} \,
|\, \xi^u_{(1+\varepsilon)\log n/c_\lambda}\neq \emptyset \}=1-o(1).\]

\end{Proposition}

Recall that for a graph $G=(\mathcal{V}_G,\mathcal{V}_E)$, the edge expansion parameter is defined as
\[\Psi_{E}(G,k)=\min_{S\subset \mathcal{V}_G, |S|\leq k}\frac{|E(S,S^c)|}{|S|},\]
where $E(S,S^c)\subset \mathcal{V}_E$ is the set of edges with one vertex in $S$ and the other vertex in $S^c$. It is shown in \cite{Hoory-Linial-Wigderson} (Theorem 4.16) that
\begin{Theorem}\label{3.expansion}
Let $d\geq 3$. Then for every $\delta>0$ there exists $\varepsilon>0$ such that for asymptotically almost every $G\sim \mathcal{G}(n,d)$, $\Psi_E(G,\varepsilon n)\geq d-2-\delta$.
\end{Theorem}

Fix an integer $M>0$. Suppose $U\subset \mathcal{V}_G$  is of size $\alpha n$, where $\alpha>0$. We remove every vertex in $U$ whose $M$-neighborhood is not a tree, and denote the remaining vertex set by $U^\prime$. We claim $|U^\prime|=\alpha n-o(n)$. Here we are using the following fact shown in \cite{Lubetzky-Sly} (Lemma 3.2): 
\begin{Proposition}\label{3.treelike}
For asymptotically almost every $G\sim \mathcal{G}(n,d)$, it has at most $o(n)$ vertices whose $\lfloor\log_{d-1}\log n\rfloor$-neighborhoods in $G$ are not tree-like.
\end{Proposition}
 Since the cardinality of $U^\prime$ and $U$ are on the same order of magnitude, without loss of generality, let us assume that all vertices in $U$ have tree-like $M$-neighborhoods in $G$. 

We classify vertices in $U$ into 2 categories by looking at their $M$-neighborhoods in $G$ in the following way. For $v\in U$, let $B(v,M)$ be the induced subgraph containing all vertices in $v$'s $M$-neighborhood in $G$. $B(v,M)\backslash \{v\}$ has $d$ connected components, call them $C_1(v), C_2(v),\dots, C_d(v)$. If $C_i(v)$ contains no other vertices in $U$, call it a \emph{free branch of depth $M$} of $v$.
\begin{itemize}
\item Color $v$ \emph{black} if at least one of $C_1(v), C_2(v),\dots, C_d(v)$ is a free branch of depth $M$ of $v$.
\item Color $v$ \emph{white} if none of $C_1(v), C_2(v),\dots, C_d(v)$ is a free branch of depth $M$ of $v$ .
\end{itemize}

\begin{Proposition}\label{3.percentage}
Fix $M\in\mathbb{N}$. There exists $\varepsilon=\varepsilon(M)>0$, such that for asymptotically almost every $G\sim\mathcal{G}(n,d)$ the following statement holds: for any set $U\subset \mathcal{V}_G$ satisfying $|U|\leq \varepsilon n$ and that every vertex in $U$ has its $M$-neighborhood in $G$ being a tree, then $U$ has at least $|U|/4$ black vertices.
\end{Proposition}
\begin{proof}
We will construct a subset of vertices $W\subset \mathcal{V}_G$.
First of all, $W$ contains all vertices in $U$. Moreover, we are going to add some
vertices into $W$ based on the white vertices of $U$.
Let $v\in U$ be a white vertex. In each of $C_1(v), C_2(v), \dots, C_d(v)$ there must be at least another vertex in $U$. Suppose $x\in U\cap C_i(v)$, then for the pair $(v,x)$, we add into $W$ every vertex along the (unique) geodesic between $v$ and $x$. We repeat this operation for every possible pair $(v,x)$ to obtain $W$.

Such constructed $W$ contains 3 types of vertices: black vertices of $U$, white vertices of $U$, and the vertices which are added by the above procedure (color them \emph{grey}). Now let us count their contributions to $E(W,W^c)$.

\begin{itemize}
\item A white vertex will contribute 0 edge to $E(W,W^c)$. This is because all of its $d$ neighboring vertices are already in $W$ by our construction.
\item A black vertex can contribute at most $d$ edges to $E(W,W^c)$, possibly fewer.
\item A grey vertex can contribute at most $d-2$ edges to $E(W,W^c)$, possibly fewer. This is because by our construction a grey vertex must be sitting on the geodesic between two other vertices in $U$ and therefore at least 2 out of its $d$ neighboring vertices are already in $W$.
\end{itemize}
Suppose in $U$ there are $w$ white vertices, $b$ black vertices. Then $g$, the number of grey vertices in $W$, satisfies $g\leq (d+d(d-1)+\dots+d(d-1)^{M-1})w:=N_M w$.

Due to Theorem \ref{3.percentage}, there exists $\varepsilon_M$ such that on a typical random regular graph $G$, $\Psi_E(G,\varepsilon_M n)\geq d-2-(3d-8)/(3N_M+4)$.
This forces the following inequality (provided $b+w+g\leq \varepsilon_M n$),
\[0w+db+(d-2)g \geq E(W,W^c)\geq (d-2-\frac{3d-8}{3N_M+4}) (w+b+g).\]
Together with $g\leq N_M w$,
we conclude that
\[\frac{b}{b+w}\geq \frac{1}{4}.\]

Therefore take $\varepsilon_M^\prime=\varepsilon_M/(N_M+1)$ (this guarantees that if $b+w\leq \varepsilon^\prime_Mn$ then $b+w+g\leq \varepsilon_M n$) . As long as $|U|\leq \varepsilon_M^\prime n$ and every vertex in $U$ has its $M$-neighborhood in $G$ being a tree, then $U$ has at least $|U|/4$  black vertices.
\end{proof}

We need the following result about the growth rate of the severed contact process on a tree, shown in \cite{Lalley-Su} (Proposition 2.2). A severed contact process $\{\eta_t^{O}\}_{t\geq 0}$ is a version of contact process on $\mathbb{T}^d$ with initial configuration $\{O\}$ where we do not allow infections to come across $d-1$ edges connected to $O$.
\begin{Proposition}\label{3.growth}
There exists $A=A(\lambda,d)>0$ such that $\mathbb{E}|\eta_t^{O}|\geq A \exp(c_\lambda t)$, for all $ t\geq 0$.
\end{Proposition}

Let $\Delta_M$ be the following finite graph. Fix the root $O$ in $\mathbb{T}^d$. We first remove $d-1$ edges connected to $O$ in $\mathbb{T}^d$. For the remaining graph, the connected component of the $M$-neighborhood of $O$ is called $\Delta_M$. 
 In $\Delta_M$, $O$ and vertices at distance $M$ from $O$ are of degree 1; other vertices are of degree $d$. $\Delta_M$ is a tree.

Let $\{\eta^{O,\Delta_M}_t\}_{t\geq 0}$ be a contact process on $\Delta_M$ with initial configuration $\{O\}$. The following corollary is easily obtained from Proposition \ref{3.growth}.

\begin{Corollary}\label{3.corollary}
For every $N>0$, there exist $T>0$ and $M\in \mathbb{N}$ such that $\mathbb{E}|\eta^{O,\Delta_M}_T|\geq N$.
\end{Corollary}

Now we can prove Theorem \ref{3.main}.

\begin{proof} [Proof of Theorem \ref{3.main}]
Let $G$ be a typical graph as in Proposition \ref{3.expansion}, \ref{3.treelike} and  \ref{3.percentage}.

In Corollary \ref{3.corollary}, take $N=10$ so that there exist $T$ and $M$ satisfying $\mathbb{E}|\eta_{T}^{O,\Delta_M}|\geq 10$.  Without loss of generality assume $T\geq 1$.  Furthermore, take $L>0$ large enough such that $\mathbb{E}\min(|\eta_{T}^{O,\Delta_M}|,L)\geq 9$. Fix such choices of $T$, $M$ and $L$.

Let $\varepsilon=\varepsilon(2M)$ be the constant in Proposition \ref{3.percentage}. As long as $U\subset \mathcal{V}_G$ is of size $\varepsilon n$, then there will be at least $\varepsilon n /2$ vertices in $U$ such that each vertex has its $2M$-neighborhood being a tree, and by Proposition \ref{3.percentage} there will be at least $\varepsilon n/8$ black vertices. We enumerate the  black vertices to be $v_1, v_2\dots, v_k$ where $k\geq \varepsilon n/8$.

Now for each $v_i$, it has one free branch of depth $2M$ (if more than one, specify one). Add $v_i$ (and the edge connected to $v_i$) to its free branch of depth $M$ (as a subgraph of the branch of depth $2M$) and we obtain a subgraph isomorphic to $\Delta_M$. So each $v_i$ is associated with a copy of $\Delta_M$. Here we use $\Delta_M$ instead of $\Delta_{2M}$ to ensure they are disjoint. For each $v_i$, we run an independent contact process on its copy of $\Delta_M$, call it $\{\eta_{t}^{v_i}\}_{t\geq 0}$. In a standard way we can couple 
$\cup_{i=1}^k \eta_{T}^{v_i}$ and $\cup_{i=1}^k \xi_{T}^{v_i}$ 
together so that $\cup_{i=1}^k \eta_{T}^{v_i}$ is always dominated by $\cup_{i=1}^k \xi_{T}^{v_i}$.

Let $X_i=\min(|\eta_{T}^{v_i}|,L)$, then $\mathbb{E}X_i\geq 9$, and $0\leq X_i\leq L$. Furthermore $(X_i)_{1\leq i\leq k}$ are i.i.d. random variables, so by Hoeffding's inequality,
\[\mathbb{P} \{ \sum_{i=1}^k (X_i -9) \leq -k \}\leq \exp\left( - \frac{2k}{L^2} \right) \leq \exp\left(-\frac{\varepsilon n}{4L^2}\right),\]
in other words, after time $T$, with probability at least $1-\exp(-\varepsilon n/(4L^2))$, we will observe at least $8k\geq \varepsilon n$ infections in $\cup_{i=1}^k \eta_{T}^{v_i}$.

To summarize, as long as we run the contact process with initial configuration whose cardinality is $\varepsilon n$, then after time $T$, with probability more than $1-\exp(-\varepsilon n/(4L^2))$, the outcome will have cardinality at least $\varepsilon n$.  Let $\beta=\varepsilon /(8L^2)$, we conclude
\[\mathbb{P}_G\{\xi^G_{\exp(\beta n)T}\neq \emptyset \}\geq 1-\exp(\beta n)\exp\left( -\frac{\varepsilon n}{4L^2}\right)=1-o(1) .\]
\end{proof}
\begin{Remark}\label{3.remark}
\textrm One can slightly change the above proof to show the following statement. There exists a constant $\varepsilon_0>0$. For every $0<\varepsilon\leq \varepsilon_0$, there exists $\beta_{\varepsilon}>0$, such that for asymptotically almost every $G\sim\mathcal{G}(n,d)$, for any $U\subset \mathcal{V}_G$ with $|U|=\varepsilon n$, we have
\[\mathbb{P}_G\{\xi^{U}_{\exp(\beta_\varepsilon n)}\neq \emptyset\}=1-o(1).\]
\end{Remark}

\section{Metastability}

Let $\tau_G=\inf \{t\geq 0: \xi_t=\emptyset \}$ be the extinction time of the contact process started from full occupancy on $G$. In this section we are devoted to proving Theorem \ref{4.main}.
The key ingredient is the following proposition in \cite{Mountford} (Proposition 2.1):
\begin{Proposition}\label{4.Mountford}
Suppose there is a sequence of graphs $G_n=(\mathcal{V}_{n}, \mathcal{E}_n)$. Let $\tau_n$ be the extinction of the contact process started from full occupancy on $G_n$. 
Also, for arbitrary $U\subset \mathcal{V}_n$, we couple $\xi^{U}_t$ and $\xi^{\mathcal{V}_n}_t$ in the standard way. 
Suppose there exist two sequences of positive real numbers $a(n)$ and $b(n)$, both tending to infinity, such that as $n\rightarrow \infty$, 
\begin{enumerate}
\item $a(n)/b(n)\rightarrow 0$;
\item $\sup_{U\subset \mathcal{V}_n}\mathbb{P}\{\xi^{U}_{a(n)}\neq \emptyset, \xi^{U}_{a(n)}\neq \xi^{\mathcal{V}_n}_{a(n)}\}\rightarrow 0$;
\item $\mathbb{P}\{\xi^{\mathcal{V}_n}_{b(n)}\neq \emptyset \}\rightarrow 1$.
\end{enumerate}
Then $\tau_n/\mathbb{E}\tau_n$ converges in distribution to an exponential distribution with mean 1.
\end{Proposition}
Now from Theorem \ref{3.main} we can take $b(n)=\exp(\beta n)$. It suffices to find an appropriate sequence $a(n)$ such that item 1 and 2 in Proposition \ref{4.Mountford} hold.
\begin{Lemma}\label{4.infect}
There exists $\gamma>0$, such that for asymptotically almost every $G\sim \mathcal{G}(n,d)$, for any two vertices $u,v$ of $G$,
\[\mathbb{P}_G\{v\in \xi^{u}_{2\log_{d-1}n}\}\geq n^{-\gamma}.\]
\end{Lemma}
\begin{proof}
By \cite{Bollobas2}, the diameter of a typical random regular graph is $(1+o(1))\log_{d-1}n$. Therefore on such a graph, for any pair of vertices $(u,v)$, the graph distance between $u$ and $v$ is no more than $(1+o(1))\log_{d-1}n$. Assume $dist(u,v)=l$, this means we can find a sequence of vertices $(w_i)_{0\leq i \leq l}$, with $w_0=u$, $w_l=v$, and that $w_i$ is connected to $w_{i+1}$ in the graph.

One way of observing $\{v\in \xi^{u}_{2\log_{d-1}n}\}$ is as follows. In time interval $[i,i+1]$ for $0\leq i \leq l-1$, we require the infection at vertex $w_i$ to go to $w_{i+1}$, and stay there alive till the end of the time interval. This will happen with probability $p\geq (1-e^{-\lambda})e^{-1}$. In time interval $[l,2\log_{d-1}n]$, we require  the infection at $v$ to stay alive. This will have probability at least $e^{-(2\log_{d-1}n-l)}$. Therefore, the overall probability is at least $p^l e^{-(2\log_{d-1}n-l)}\geq (1-e^{-\lambda})^{2\log_{d-1}n} e^{-2\log_{d-1}n}\geq n^{-\gamma}$ where
$\gamma=2/\log(d-1)+2\log(e^\lambda / (e^{\lambda}-1))/\log(d-1)$.
\end{proof}

\begin{Lemma}\label{4.growth}
There exists $\delta_0>0$, such that for asymptotically almost every $G\sim \mathcal{G}(n,d)$, for any $v\in \mathcal{V}_G$,
\[\mathbb{P}_G\{|\xi^{v}_{n^{2\gamma}}|\geq \delta_0 n \,|\, \xi^{v}_{n^{2\gamma}} \neq \emptyset \}=1-o(n^{-2}).\]
\end{Lemma}
\begin{proof}
For asymptotically almost every $G\sim \mathcal{G}(n,d)$, there are $n-o(n)$ good vertices by Proposition \ref{2.typical survival}. In particular, there is at least one good vertex. 
We fix one good vertex in $G$, call it $w$. Our idea is as follows. We chop the time interval $[0,n^{2\gamma}]$ into $[0,T], [T,2T], \dots, [(M-1)T,MT]$, where $M=\lfloor n^{2\gamma}/T\rfloor$ and $T$ to be specified. We will construct an event and see if it happens in each interval $[iT,(i+1)T]$. The chance that it happens in each time interval is on the order of $n^{-\gamma}$. This event is constructed such that as long as in at least one of these intervals the event happens, then with probability approaching 1 at the end we will see $\delta_0 n$ infections.

Here are more details. Let $T=2\log_{d-1}n+(1+\varepsilon)\log n/c_\lambda$, where this $\varepsilon>0$ is the same as in Proposition \ref{2.typical survival} and \ref{2.typical survival size}. We say that we observe a success in time interval $[(i-1)T,iT]$ if the following two events happen.
\begin{enumerate}
\item $w\in \xi^v_{(i-1)T+2\log_{d-1}n}$ ;
\item  $|\xi^v_{iT}|\geq np_{\lambda} /2$.
\end{enumerate}
Conditional on $\xi^v_{n^{2\gamma}}\neq \emptyset$, for sure $\xi^v_{(i-1)T}\neq \emptyset$, so by Lemma \ref{4.infect}, $\{w\in \xi^v_{(i-1)T+2\log_{d-1}n}\}$ happens with probability at least $n^{-\gamma}$ (notice that the event $\{w\in \xi^v_{(i-1)T+2\log_{d-1}n}\}$ is positively correlated with the event $\{\xi^v_{n^{2\gamma}}\neq \emptyset \}$).
Given $\{w\in \xi^v_{(i-1)T+2\log_{d-1}n}\}$, since $w$ is a good vertex, by Proposition \ref{2.typical survival size}, $|\xi^v_{iT}|\geq np_{\lambda} /2$ will happen with probability at least $p_{\lambda}/2$. 

Therefore overall we have order $(n^{2\gamma}/ \log n)$ trials, each with success probability at least $p_{\lambda}n^{-\gamma}/2$, so it is easy to conclude that the chance of having at least 1 success is $1-o(n^{-2})$.
Given that  we observe a success, which means that we observe $p_\lambda n /2$ infections at some time between $[0, n^{2\gamma}]$, from the proof of Theorem \ref{3.main} and Remark \ref{3.remark}, we know that the chance of  $p_{\lambda} n/2$ infections not lasting exponentially long time before the size of infections shrinks to $\delta_0n$ is exponentially small in $n$, where $\delta_0$ can be taken as $\min(\varepsilon_0,p_{\lambda}/2)$.
Therefore the overall probability is $(1-o(n^{-2}))\times (1-o(e^{-\beta n}))= 1-o(n^{-2})$.
\end{proof}

We call a sequence of vertices $v_0,\dots, v_L\in \mathcal{V}_G$ a \emph{path} of length $L$ if $v_iv_{i+1}\in \mathcal{E}_G$ for $0\leq i \leq L-1$. We say such path has endpoints $v_0$ and $v_L$. A path of length 0 is allowed, where the endpoints are identical.
\begin{Lemma}\label{4.path}
For every $r>0$, for asymptotically almost every $G\sim \mathcal{G}(n,d)$, there exist $L_r\in \mathbb{N}$ and $c_r>0$, such that for any  $U,W\subset \mathcal{V}_G$ with $|U|=|W|=rn$, there exits $c_r n$ paths in $G$ so that
\begin{itemize}
\item Each path is of length no more than $L_r$.
\item Each path has one endpoint in $U$ and the other endpoint in $W$.
\item None of the paths share the same vertex.
\end{itemize}
\end{Lemma}
\begin{proof}
From  \cite{Bollobas3} and \cite{Kolesnik-Wormald}, there exits $h_d>0$ so that for a typical d-random regular graph $G$ its (vertex) isoperimetric constant is at least $h_d$.
Without loss of generality, we assume $r<1/2$. For $i \in \mathbb{N}$, 
Let $U_i=\{v\in \mathcal{V}_G: dist(v,U)\leq i\}$ and $W_i=\{v\in \mathcal{V}_G: dist(v,W)\leq i\}$. By the definition of the isoperimetric constant, $U_1$ has cardinality at least $(1+h_d)rn$, and that $U_{i+1}$ has cardinality at least $(1+h_d)|U_{i}|$, provided that $|U_i|\leq n/2$. The same holds for $W_i$. 

Take $L_r=2\lceil\log_{1+h_d}(1/(2r))\rceil+2$, we claim that $|U_{L_r/2}|\geq \frac{d+2h_d}{2d+2h_d}n$. This is because from the way we choose $L_r$, for sure there will be some $k\leq L_r/2-1$ so that $k$ is the smallest integer satisfying $|U_k|\geq n/2$. If $|U_k|\geq \frac{d+2h_d}{2d+2h_d}n$ then we are done. Otherwise, since $|U^c_k|\leq n/2$, the number of vertices in $U_k$ that are adjacent to $U^c_k$ is at least $h_d|U^c_k|$, and that the number of vertices in $U^c_k$ that are adjacent to $U_k$ is at least $h_d|U^c_k|/d$. In this case
\[|U_{k+1}|\geq |U_k|+\frac{h_d}{d}|U_k^c|. \]
Since $|U_k|\geq n/2$ and $|U^c_k|\geq n-\frac{d+2h_d}{2d+h_d}n$, we have
\[|U_{k+1}|\geq\frac{n}{2}+\frac{h_d}{d}\left(n-\frac{d+2h_d}{2d+2h_d}n\right)=\frac{d+2h_d}{2d+2h_d}n.\]

Now we are able to find $|U_k|,|W_j|\geq \frac{d+2h_d}{2d+2h_d}n$, where $k,j\leq L_r/2$. By the inclusion-exclusion principle, 
\[|U_k\cap W_j|\geq\frac{h_d}{d+h_d}n.\]
Now each vertex $v\in U_k\cap W_j$ naturally corresponds to a path of length no more than $L_r/2+L_r/2=L_r$ with one endpoint in $U$ and the other in $W$. Notice that there exits $M_{L_r}\in \mathbb{N}$, so that any path of length no more than $L_r$ can intersect at most $M_{L_r}$ other paths of length no more than $L_r$ in $G$, therefore we can pick at least
\[\lfloor \frac{1}{M_{L_r}}\frac{h_d}{d+h_d}n\rfloor\]
non-intersecting paths which satisfy all requirements stated in the lemma.
\end{proof}

\begin{Proposition}\label{4.coupling}
Let $a(n)=2n^{2\gamma}+L_{\delta_0}$, where $\delta_0$ is as in Lemma \ref{4.growth} and $L_{\delta_0}$ is as in Lemma \ref{4.path}. Then for asymptotically almost every $G\sim \mathcal{G}(n,d)$, 
\[\sup_{U\subset \mathcal{V}_G}\mathbb{P}_G\{\xi^{U}_{a(n)}\neq \emptyset, \xi^{U}_{a(n)}\neq \xi^{\mathcal{V}_G}_{a(n)}\}=o(1).\]
\end{Proposition}
\begin{proof}
It suffices to show
\begin{equation}\label{4.supbound}
\sup_{v\in \mathcal{V}_G}\mathbb{P}_G\{\xi^{v}_{a(n)}\neq \emptyset, \xi^{v}_{a(n)}\neq \xi^{\mathcal{V}_G}_{a(n)} \}=o(n^{-1}),
\end{equation}
because for arbitrary $U\subset \mathcal{V}_G$,
\[\mathbb{P}_G \{\xi^{U}_{a(n)}\neq \emptyset, \xi^{U}_{a(n)}\neq \xi^{\mathcal{V}_G}_{a(n)} \}  \leq \sum_{v\in U}\mathbb{P}_G \{\xi^{v}_{a(n)}\neq \emptyset, \xi^{v}_{a(n)}\neq \xi^{\mathcal{V}_G}_{a(n)} \}   .\]

Moreover, in order to show (\ref{4.supbound}) it suffices to show the following inequality by a union bound,
\begin{equation}\label{4.unionbound}
\sup_{v\in \mathcal{V}_G,u\in \mathcal{V}_G}\mathbb{P}_G\{\xi^{v}_{a(n)}\neq \emptyset, u\notin \xi^{v}_{a(n)}, u\in \xi^{\mathcal{V}_G}_{a(n)}\}=o(n^{-2}).
\end{equation}

Recall the graphical representation and the dual contact process in \cite{Liggett:book1}. To show (\ref{4.unionbound}), we chop the time interval $[0,a(n)]$ into $[0, n^{2\gamma}+L_{\delta_0}]$ and $[n^{2\gamma}+L_{\delta_0}, a(n)]$. We run $\xi^{v}_t$ for time $n^{2\gamma}+L_{\delta_0}$ in the first subinterval and run the dual process $\hat{\xi}^u_t$   for time $n^{2\gamma}$ in the second subinterval, where time $t$ for the dual process corresponds to time $a(n)-t$ for the original process. In particular, time $0$ for the dual process corresponds to time $a(n)$ for the original one, and time $n^{2\gamma}$ for the dual process corresponds to time $n^{2\gamma}+L_{\delta_0}$ for the original one. 

Notice that $\xi_t^v$ and $\hat{\xi}_t^u$ are two independent processes. Then we have the following upper bound of (\ref{4.unionbound}),
\[\begin{split}
&\mathbb{P}_G\{\xi^{v}_{a(n)}\neq \emptyset, u \notin \xi^{v}_{a(n)}, u\in \xi^{\mathcal{V}_G}_{a(n)}\}\leq  \mathbb{P}_G\{\xi^v_{n^{2\gamma}}\neq \emptyset, \hat{\xi}^u_{n^{2\gamma}}\neq \emptyset, \xi^{v}_{n^{2\gamma}+L_{\delta_0}}\cap \hat{\xi}^{u}_{n^{2\gamma}}=\emptyset\}  \\
\leq &\mathbb{P}_G\{\xi^{v}_{n^{2\gamma}+L_{\delta_0}}\cap \hat{\xi}^{u}_{n^{2\gamma}}=\emptyset\,|\,\xi^v_{n^{2\gamma}}\neq \emptyset, \hat{\xi}^u_{n^{2\gamma}}\neq \emptyset\}.
\end{split}\]

Our goal is to show the following bound holds uniformly for $u,v$: 
\begin{equation}\label{4.conditionalbound}
\mathbb{P}_G\{\xi^{v}_{n^{2\gamma}+L_{\delta_0}}\cap \hat{\xi}^{u}_{n^{2\gamma}}=\emptyset\,|\,\xi^v_{n^{2\gamma}}\neq \emptyset, \hat{\xi}^u_{n^{2\gamma}}\neq \emptyset\}=o(n^{-2}).
\end{equation}

By Lemma \ref{4.growth},
\begin{equation}\label{4.sizebound}
\mathbb{P}_G\{|\xi^{v}_{n^{2\gamma}}|\geq \delta_0 n, |\hat{\xi}^{u}_{n^{2\gamma}}|\geq \delta_0 n\,|\,\xi^v_{n^{2\gamma}}\neq \emptyset, \hat{\xi}^u_{n^{2\gamma}}\neq \emptyset\}=(1-o(n^{-2}))^2=1-o(n^{-2}).
\end{equation}
Given $\{|\xi^{v}_{n^{2\gamma}}|\geq \delta_0 n, |\hat{\xi}^{u}_{n^{2\gamma}}|\geq \delta_0 n\}$, by Lemma \ref{4.path}, there will be $c_{\delta_0} n$ non-intersecting paths of length no more than $L_{\delta_0}$ with one endpoint in $\xi^{v}_{n^{2\gamma}}$ and the other in $\hat{\xi}^{u}_{n^{2\gamma}}$.
We call it a \emph{success} on a path $v_0,v_1,\dots,v_k$ with endpoints $v_0\in \xi^{v}_{n^{2\gamma}}$ and $v_k\in \hat{\xi}^{u}_{n^{2\gamma}}$, if in time $L_{\delta_0}$, the infection at $v_0$ spreads to $v_k$ within this path. For non-intersecting paths, successes on paths are independent events. Also, using similar argument as in proof of Lemma \ref{4.growth}, it is easy to see that there exits $p_{\delta_0}>0$ so that the probability of a success is at least $p_{\delta_0}$ on a path of length no more than $L_{\delta_0}$.

In order to observe $\{\xi^{v}_{n^{2\gamma}+L_{\delta_0}}\cap \hat{\xi}^{u}_{n^{2\gamma}}\neq \emptyset\}$, it suffices that on one of these $c_{\delta_0}n$ paths we observe a success, each having probability at least $p_{\delta_0}>0$. Therefore the chance of observing at least 1 success is at least $1-o(n^{-2})$ by an easy binomial calculation.

Combining all arguments above, we see
\[\begin{aligned}
&\mathbb{P}_G\{\xi^{v}_{n^{2\gamma}+L_{\delta_0}}\cap \hat{\xi}^{u}_{n^{2\gamma}}\neq \emptyset\,|\,\xi^v_{n^{2\gamma}}\neq \emptyset, \hat{\xi}^u_{n^{2\gamma}}\neq \emptyset\}\\
\geq &   \mathbb{P}_G\{|\xi^{v}_{n^{2\gamma}}|\geq \delta_0 n, |\hat{\xi}^{u}_{n^{2\gamma}}|\geq \delta_0 n\,|\,\xi^v_{n^{2\gamma}}\neq \emptyset, \hat{\xi}^u_{n^{2\gamma}}\neq \emptyset\} \times (1-o(n^{-2}))\\
=& (1-o(n^{-2}))^2  \\
=&1-o(n^{-2}),
\end{aligned}
\]
which finishes our proof.
\end{proof}
\begin{proof}[Proof of Theorem \ref{4.main}]
By Proposition \ref{4.Mountford}, take $b(n)=\exp(\beta n)$ as in Theorem \ref{3.main} and $a(n)=2n^{2\gamma}+L_{\delta_0}$ as in Proposition \ref{4.coupling}.
\end{proof}

\section*{Acknowledgements}
The author would like to thank his advisor Professor Steven Lalley for suggesting this problem and many useful discussions.
\bibliographystyle{plain}
\bibliography{mainbib}

\end{document}